\documentclass{amsart}
\usepackage{lineno,hyperref}
\usepackage{amssymb,amsmath}
\usepackage{amsthm}
\usepackage{mathrsfs}
\usepackage{enumitem, indentfirst}
\usepackage[fleqn,tbtags]{mathtools}
\usepackage{color}
\modulolinenumbers[5]

 \newtheorem{theorem}{Theorem}[section]
 \newtheorem{corollary}[theorem]{Corollary}

 \theoremstyle{definition}
 \newtheorem{definition}[theorem]{Definition}

 \theoremstyle{remark}
  \numberwithin{equation}{section}

\renewcommand{\phi}{\varphi}
\renewcommand{\theta}{\vartheta}

\DeclareMathOperator{\tform}{\mathfrak{t}}
\DeclareMathOperator{\tfrom}{\mathfrak{t}}

\DeclareMathOperator{\wform}{\mathfrak{w}}

\DeclarePairedDelimiterX\sipt[2]{(}{)_{\tform}}{#1\,\delimsize\vert\,#2}
\DeclarePairedDelimiterX\sipv[2]{(}{)_{v}}{#1\,\delimsize\vert\,#2}
\DeclarePairedDelimiterX\sipw[2]{(}{)_{w}}{#1\,\delimsize\vert\,#2}

\newcommand{\alg}{\mathscr{A}}

\newcommand{\dupN}{\mathbb{N}}

\newcommand{\seq}[1]{(#1_{n})_{n\in\dupN}}

\newcommand{\dupC}{\mathbb{C}}

\newcommand{\pif}{\pi_f}

\newcommand{\dom}{\operatorname{dom}}

\newcommand{\ran}{\operatorname{ran}}

\newcommand{\lef}{\mathscr{L}(E;F)}

\newcommand{\D}{\mathscr{D}}
\newcommand{\DD}{\bar{\mathscr{D}}'}

\newcommand{\sigef}{\sigma(E,F)}
\newcommand{\sigfe}{\sigma(F,E)}
\newcommand{\hil}{H}

\newcommand{\hilf}{\hil_f}
\newcommand{\hila}{H_A}

\newcommand{\bh}{\mathscr{B}(\hil)}
\newcommand{\bhf}{\mathscr{B}(\hilf)}

\DeclarePairedDelimiterX\abs[1]{\lvert}{\rvert}{#1}
\DeclarePairedDelimiterX\sip[2]{(}{)}{#1\,\delimsize\vert\,#2}
\DeclarePairedDelimiterX\siptilde[2]{(}{)_{\!_{\widetilde{A}}}}{#1\,\delimsize\vert\,#2}
\DeclarePairedDelimiterX\sipf[2]{(}{)_{f}}{#1\,\delimsize\vert\,#2}
\DeclarePairedDelimiterX\sipg[2]{(}{)_{g}}{#1\,\delimsize\vert\,#2}
\DeclarePairedDelimiterX\siptw[2]{(}{)_{\tform+\wform}}{#1\,\delimsize\vert\,#2}
\DeclarePairedDelimiterX\set[2]{\{}{\}}{#1\,:\,#2}
\DeclarePairedDelimiterX\dual[2]{\langle}{\rangle}{#1,#2}
\DeclarePairedDelimiterX\sipa[2]{(}{)_{\!_A}}{#1\,\delimsize\vert\,#2}
\DeclarePairedDelimiterX\sipc[2]{(}{)_{\!_C}}{#1\,\delimsize\vert\,#2}
\DeclarePairedDelimiterX\sipab[2]{(}{)_{\!_{A+B}}}{#1\,\delimsize\vert\,#2}
\DeclarePairedDelimiterX\sipb[2]{(}{)_{\!_B}}{#1\,\delimsize\vert\,#2}
\newcommand{\anti}[1]{\bar{#1}^*}

\newcommand{\limn}{\lim\limits_{n\rightarrow\infty}}

\newcommand{\kismatrix}[4]{\begin{bsmallmatrix} #1 &  #2\\ #3& #4\end{bsmallmatrix}}
\newcommand{\kispair}[2]{\begin{bsmallmatrix} #1 \\  #2\end{bsmallmatrix}}

\allowdisplaybreaks

\begin{document}

%\begin{frontmatter}

\title{Operators on anti-dual pairs:\\ Generalized Schur complement}
%\tnotetext[mytitlenote]{Corresponding author: T. Titkos}

\author{Zsigmond Tarcsay}
\address{Zs.~Tarcsay \\ Department of Applied Analysis  and Computational Mathematics\\ E\"otv\"os Lor\'and University, P\'azm\'any P\'eter s\'et\'any 1/c., Budapest H-1117, Hungary}
\thanks{The corresponding author Zs. Tarcsay was supported by DAAD-TEMPUS Cooperation Project ``Harmonic Analysis and Extremal Problems'' (grant no. 308015). Project no. ED 18-1-2019-0030 (Application-specific highly reliable IT
solutions)
has been implemented with the support provided from the National Research,
Development and Innovation Fund of Hungary, financed under the Thematic
Excellence
Programme funding scheme.}
\author{Tam\'as Titkos}
\address{T. Titkos\\ Alfr\'ed R\'enyi Institute of Mathematics\\ Re\'altanoda utca 13-15., Budapest H-1053, Hungary\\ and\\ BBS University of Applied Sciences\\
Alkotm\'any u. 9., Budapest H-1054, Hungary}
\thanks{
T. Titkos was supported by the Hungarian National Research, Development and Innovation Office  NKFIH (grant no. PD128374 and grant no. K115383), by the J\'anos Bolyai Research Scholarship of the Hungarian Academy of Sciences, and by the \'UNKP-19-4-BGE-1 New National Excellence Program of the Ministry for Innovation and Technology.}

\begin{abstract} The goal of this paper is to develop the theory of Schur complementation in the context of operators acting on anti-dual pairs. As a byproduct, we obtain a natural generalization of the parallel  sum and parallel difference, as well as the Lebesgue-type decomposition. To demonstrate how this operator approach works in application, we derive the corresponding results for operators acting on rigged Hilbert spaces, and for representable functionals of ${}^{*}$-algebras. 
\end{abstract}

\keywords{Positive operator, anti-duality, Schur complement, parallel sum, parallel difference, rigged Hilbert space, $^*$-algebra}

\subjclass[2010]{Primary 47B65, Secondary 28A12, 46K10}
%AMS Subject classification: Primary 47B65, Secondary 28A12, 46K10}

%\end{keyword}

%\end{frontmatter}
\maketitle

\section{Introduction}
Since the first appearance of the name of the ``Schur complement'' in \cite{S2}, the theory of partitioned matrices (or block operators) is an active field of research in linear algebra and functional analysis. The direction we are interested in is the problem of completing special operator systems. To formulate the central question in the most classical setting, consider the incomplete system $\mathfrak{S}=\kismatrix{A}{B}{B}{*}$ of positive semidefinite $n\times n$ matrices $A,B$. The task is to find a matrix $D$ for which $\kismatrix{A}{B}{B}{D}$ is a positive semidefinite $2n\times 2n$ matrix. 
If we denote by $A_B$ the smallest possible solution, then the Schur complement  of $D$ in the block-matrix $\kismatrix{A}{B}{B}{D}$ is $D-A_B$. Therefore, to find the Schur complement and to find the minimal operator that makes a system positive is the same problem. In this paper, we focus our attention to the completion problem.

Because of its wide-range applicability in pure and applied mathematics, a number of authors made a lot of efforts to extend the concept of Schur complement for various settings.  We mention first the fundamental work of Pekarev and \v{S}mul'jan \cite{Pekarev} on the connection between the shorted operator and positive completions of block operators in the context of Hilbert spaces. The corresponding result in Krein spaces has been developed by Contino, Maestripieri, and Marcantognini in \cite{CMM} (see also \cite{B,MMP}). The relation between extension, completion, and lifting problems of operators on both Hilbert and Krein spaces has been discussed in \cite{BaidiukHassi,B}. 
A quite general approach was developed  by Friedrich, G\"unther and Klotz in \cite{FGK}. They introduced a generalized Schur complement for non-negative $2\times 2$ block matrices whose entries are linear operators on linear spaces. 
In their considerations the setting is purely algebraic and therefore topology plays a minor role. 

In the present paper we are going to treat the above completion problem in an even more general setting that covers Hilbert, Krein, and linear spaces. Namely, we consider linear operators acting between a locally convex space and its topological anti-dual. The key idea of our approach is the observation that the block matrix completion problem can be formulated as an operator extension problem. This gives rise to invoke our corresponding Krein--von Neumann extension theory developed in \cite{KV}. Our aim is two folded: besides of solving the block completion problem in a quite general setting, we want to demonstrate how the developed method can be applied for structures like rigged Hilbert spaces and involutive algebras.

The paper is organized as follows: after collecting the basic notions and notations regarding operators acting on anti-dual pairs, in Subsection \ref{Ss:Krein} we recall the construction of the generalized Krein--von Neumann extension. 

Subsection \ref{Ss:Schur}, the cornerstone of this paper, is devoted to provide necessary and sufficient conditions to guarantee that an incomplete operator system is positive. In case of positivity, Theorem \ref{T:Schur} gives an explicit formula for the minimal solution of the completion problem. 

Following the method of Pekarev and \v{S}mul'jan, we introduce the notions of parallel sum and difference as an immediate application. This will be done in Subsection \ref{Ss:PsPdLeb}. Furthermore, by means of these notions we extend a very recent result appeared in \cite{Tarcsay-ADP Lebdec} on Lebesgue-type decompositions. Namely, in  Theorem \ref{T:ADP-Lebdec} we exhibit an alternative description of the absolutely continuous part. 

Section \ref{S:Appl} is fully devoted to applications: to demonstrate how this operator approach works in concrete structures, we derive the corresponding results for operators acting on rigged Hilbert spaces and for representable functionals on involutive algebras. 

\section{Positive completions of operator systems}

\subsection{Operators on anti-dual pairs} \label{Ss:Krein}
Throughout this paper we follow the notations of \cite{KV}. For more details we refer the reader to \cite[Section 2 and 3]{KV}.
An anti-dual pair denoted by $\dual FE$ is a pair of two complex linear spaces intertwined by a separating sesquilinear map $\dual\cdot\cdot:F\times E\to\dupC$.

If $\mathcal{D}$ is a linear subspace of $E$, a linear operator $A:\mathcal{D}\to F$ is said to be \emph{positive}, if $\dual{Ax}{x}\geq 0$ for all $x\in\mathcal{D}$. In this paper we always assume $E$ and $F$ to be endowed with the corresponding weak topologies $\sigef$, and $\sigfe$, respectively. 
We will call the anti-dual pair $\dual FE$ weak-* sequentially complete if the topological vector space $(F,\sigfe)$ is sequentially complete. Recall that the (topological) anti-dual space $\anti{E}$ of a barrelled space $E$ is quasi-complete, hence the anti-dual pair $\dual{\anti{E}}{E}$ is a weak-* sequentially complete. It is furthermore obvious that the algebraic anti-dual $\bar E'$  is  weakly (sequentially) complete, hence the class of weak-* sequentially complete anti-dual pairs includes Hilbert, Banach, Fr\'echet spaces, and also vector spaces without topology.

Let $\dual{F_1}{E_1}$ and $\dual{F_2}{E_2}$ be anti-dual pairs and  $T:E_1\to F_2$ a weakly continuous (that is, $\sigma(E_1,F_1)$-$\sigma(F_2,E_2)$-continuous) linear operator. Then the weakly continuous linear operator $T^*:E_2\to F_1$ satisfying \begin{equation*}
    \dual{Tx_1}{x_2}_2=\overline{\dual{T^*x_2}{x_1}_1},\qquad  x_1\in E_1,x_2\in E_2
\end{equation*} 
is called the adjoint of $T$. The set of \emph{everywhere defined} weakly continuous linear operators $T:E\to F$ is denoted by $\lef$. Just like in the case of Hilbert spaces, every operator $T\in\lef$ is uniquely determined by its quadratic form $x\mapsto\dual{Tx}{x}.$

As it plays an important role in the background, we recall the construction of the Krein-von Neumann extension of a positive operator. For more details see \cite[Theorem 3.1]{KV}. Assume that $\dual{F}{E}$ is a $w^*$-sequentially complete anti-dual pair and  $A:\mathcal{D}\to F$ is a linear operator such that
for any $y$ in $E$ there is $M_y\geq0$ satisfying 
\begin{equation}\label{E:M_y}
 \abs{\dual{Ax}{y}}^2\leq M_y\dual{Ax}{x}\qquad \textrm{for all $x\in\mathcal{D}$.}
\end{equation}
This assumption guarantees that one can build a Hilbert space $\hil_{A}$ by taking the Hilbert space completion of the inner product space $\big(\ran A,\sipa{\cdot}{\cdot}\big)$, where
\begin{equation}
\sipa{{A}x}{{A}x'}:=\dual{{A}x}{x'},\qquad x,x'\in\dom {A}.
\end{equation}
Again, by \eqref{E:M_y}, the canonical embedding operator 
\begin{equation}\label{E:J0}
    J_0:H_A\supseteq\ran A\to F,\qquad J_0({A}x):={A}x
\end{equation}
is weakly continuous, and thus admits a unique continuous extension $J$ to $\hila$ by weak-$^*$ sequentially completeness. Since $J\in\mathscr{L}(\hil_{A} ;F)$, its adjoint $J^*$ belongs to $\mathscr{L}(E ;\hil_{A})$, and $J^*x={A}x$ for all $x\in\dom A$. As for any $x\in\dom {A}$ we have $JJ^*x=J({A}x)={A}x$, 
the operator $JJ^*\in\lef$ is a positive extension of $A$. We will refer to $A_N:=JJ^*$ as the Krein-von Neumann extension of $A$. For $A_N$ we obtained the following formulae (see (3.6) and (3.7) in \cite{KV}) 
\begin{align}\label{L:JJquadratic}
\dual{JJ^*y}{y}&=\sup\set[\big]{\abs{\dual{Ax}{y}}^2}{x\in\dom A, \dual{Ax}{x}\leq 1}\\
               &=\sup\set[\big]{\dual{Ax}{y}+\overline{\dual{Ax}{y}}-\dual{Ax}{x}}{x\in\dom A}
\label{L:JJquadratic2}
\end{align}
\subsection{Generalized Schur complement}\label{Ss:Schur}

Using the technique introduced in the previous subsection, we can extend the corresponding results of Pekarev and Smul'jan \cite{Pekarev}. Namely, for given operators $A,B,C\in \lef$ we consider the \emph{incomplete operator matrix} (or shortly \emph{system}) $\kismatrix{A}{B}{C}{*}$. Such a system is called positive if there exists a  $D\in\lef$ such that $\kismatrix{A}{B}{C}{D}$ is positive. Clearly, for an incomplete operator matrix to be positive it is necessary that $A\geq0$ and that $C=B^*$, furthermore  every $D$ making  $\kismatrix{A}{B}{C}{*}$ positive  must be positive itself. We may therefore restrict ourselves to investigate incomplete operator matrices of the form $\kismatrix{A}{B}{B^*}{*}$,  where $A\geq 0$. The next theorem provides necessary and sufficient conditions for positivity.

\begin{theorem}\label{T:Schur}
Let $\dual{F_1}{E_1}$ and $\dual{F_2}{E_2}$ be weak-$^*$ sequentially complete anti-dual pairs and let  $A\in \mathscr{L}(E_1;F_1)$ and $B\in\mathscr{L}(E_1;F_2)$ be weakly continuous linear operators such that $A\geq0$. Then the following assertions are equivalent: 
\begin{enumerate}[label=\textup{(\roman*)}]
 \item There is a positive operator $C\in\mathscr{L}(E_2;F_2)$ such that the operator matrix $\kismatrix{A}{B^*}{B}{C}$ is positive.
\item For every $y_2\in E_2$ there exists a constant $M_{y_2}\geq0 $ such that
 \begin{align*}
 \abs{\dual{Bx_1}{y_2}}^2\leq M_{y^{}_2}\cdot \dual{Ax_1}{x_1}\qquad \textrm{for all $x_1\in E_1$.}
\end{align*}
\item   For the canonical embedding operator $J:\hila\to F_1$ constructed in \eqref{E:J0}
the following range inclusion holds
\begin{align*}\ran B^*\subseteq \ran J.
\end{align*}
\end{enumerate}
If any of the above conditions is fulfilled, then the linear operator 
\begin{equation}\label{E:S_0-def}
S_0:\hila\supseteq \ran A\to F_2;\qquad  Ax_1\mapsto Bx_1,\qquad x_1\in E_1
\end{equation}
is well defined and weakly continuous. Furthermore, its unique continuous extension $S\in \mathscr{L}(\hila; F_2)$ possesses the property that 
\begin{equation}\label{Schurdef}
A^{}_B\coloneqq SS^*
\end{equation}
is the smallest positive operator that makes $\kismatrix{A}{B^*}{B}{*}$ positive. The quadratic form of $A^{}_B$ is given by  
\begin{align}\label{C:Cor32}
\dual{SS^*y_2}{y_2}&=\sup\set[\big]{\abs{\dual{Bx_1}{y_2}}^2 }{x_1\in E_1, \dual{Ax_1}{x_1}\leq 1}\\
               &=\sup\set[\big]{\dual{Bx_1}{y_2}+\dual{B^*y_2}{x_1}-\dual{Ax_1}{x_1} }{x_1\in E_1}.\label{2cof}
\end{align}
\end{theorem}
\begin{proof}
We use the column vector notation $\kispair{x_1}{x_2}$ instead of $(x_1,x_2)$. 

(i)$\Rightarrow$(ii): 
 Take $y_2$ from $E_2$. By the Cauchy-Schwarz inequality we have 
\begin{align*}
\abs{\dual{Bx_1}{y_2}}^2&=\abs[\big]{\dual[\big]{\kismatrix{A}{B^*}{B}{C}\kispair{x_1}{0^{}_{}}}{\kispair{0}{y_2}}}^2\\ 
&\leq \dual[\big]{\kismatrix{A}{B^*}{B}{C}\kispair{0}{y_2}}{\kispair{0}{y_2}}\dual[\big]{\kismatrix{A}{B^*}{B}{C}\kispair{x_1}{0^{}_{}}}{\kispair{x_1}{0^{}_{}}}\\
&=\dual[\big]{\kismatrix{A}{B^*}{B}{C}\kispair{0}{y_2}}{\kispair{0}{y_2}}\dual{Ax_1}{x_1}
\end{align*}
 for all $x_1$ in $E_1$. Hence (i) implies (ii).
 
 (ii)$\Rightarrow$ (iii): Consider the auxiliary Hilbert space $\hila$ and for a fixed vector $y_2\in E_2$ and define the linear functional
 \begin{equation*}
    \phi:\hila\supseteq \ran A\to \dupC;\qquad \phi(Ax_1)\coloneqq \dual{Bx_1}{y_2},\qquad x_1\in E_1.
 \end{equation*}
 From (ii) we  infer that $\phi$ is bounded and therefore the Riesz representation theorem implies that there exists $h\in\hila$ such that 
 \begin{equation*}
     \sipa{Ax_1}{h}=\dual{Bx_1}{y_2},\qquad x_1\in E_1.
 \end{equation*}
 Denote by $J$ the canonical embedding operator \eqref{E:J0} of $\hila$ into $F_1$. Since we have $J^*x_1=Ax_1$, it follows that 
 \begin{equation*}
     \dual{Jh}{x_1}=\sipa{h}{Ax_1}=\overline{\dual{Bx_1}{y_2}}=\dual{B^*y_2}{x_1}
 \end{equation*}
 for every $x_1\in E_1$. Consequently, $B^*y_2=Jh\in\ran J$ which proves (iii).
 
 (iii)$\Rightarrow$(i): We are going to show that the operator $T_0:E_1\times \{0\}\to F_1\times F_2,$
 \begin{equation}\label{E:M0-def}
      T_0\kispair{x_1}{0}\coloneqq \kispair{Ax_1}{Bx_1} 
 \end{equation}
 has a positive extension $T\in\mathscr{L}(E_1\times E_2;F_1\times F_2)$. Since every such extension can be written as a matrix of the form $T=\kismatrix{A}{B^*}{B}{C}$ for some $C\in\mathscr{L}(E_2;F_2)$, $C\geq0,$ this will entail the desired implication.
 
 The anti-dual pair $\dual{F_1\times F_2}{E_1\times E_2}$ is obviously weak-* sequentially complete, so our only duty is the verify that $T_0$ satisfies condition 
 \eqref{E:M_y}, as it guarantees the existence of the Krein-von Neumann extension according to \cite[Theorem 3.1]{KV}.
To this end, fix $\kispair{y_1}{y_2}\in E_1\times E_2$, then 
\begin{align*}
\abs{\dual{T_0\kispair{x_1}{0}}{\kispair{y_1}{y_2}}}^2&=\abs{\dual{\kispair{Ax_1}{Bx_1}}{\kispair{y_1}{y_2}}}^2\\&=\abs{\dual{Ax_1}{y_1}+\dual{Bx_1}{y_2}}^2\\
&\leq 2\abs{\dual{Ax_1}{y_1}}^2+2\abs{\dual{Bx_1}{y_2}}^2\\
&\leq 2\dual{Ax_1}{x_1}\dual{Ay_1}{y_1}+2M_{y_2}\dual{Ax_1}{x_1}\\
&=M'_{y_1,y_2}\cdot \dual{\kispair{Ax_1}{Bx_1}}{\kispair{x_1}{0^{}_{}}}=M'_{y_1,y_2}\cdot \dual{T_0\kispair{x_1}{0}}{\kispair{x_1}{0}}
\end{align*}
for all $\kispair{x_1}{0}$ in $\dom T_0$.
To prove the rest of our statement observe first that inequality (ii) can be written in the form 
\begin{equation*}
    \abs{\dual{Bx_1}{y_2}}^2\leq M_{y_2}\cdot \sipa{Ax_1}{Ax_1},\qquad x_1\in E, y_2\in E_2.
\end{equation*}
This inequality says that for every fixed $y_2\in E_2$, the linear functional 
\begin{equation*}
    Ax_1\mapsto \dual{Bx_1}{y_2}
\end{equation*}
from $\ran A\subseteq \hila$ into $\dupC$ is continuous. This implies  that $S_0$ of \eqref{E:S_0-def} is well defined and weakly continuous. Let $T_N$ denote the smallest (Krein-von Neumann) extension of $T_0$ in \eqref{E:M0-def}. If we write $T_N$ in the matrix form $T_N=\kismatrix{A}{B^*}{B}{C_N}$,   $C_N\in \mathscr L(E_2;F_2)$ is evidently the smallest positive operator that makes $\kismatrix{A}{B^*}{B}{*}$ positive. It is therefore enough to show that $C_N=SS^*$.  To this aim, take $y_2\in E_2$.
From \eqref{L:JJquadratic} it follows that
\begin{align*}
\dual{C_Ny_2}{y_2}&=\dual{T_N\kispair{0}{y_2}}{\kispair{0}{y_2}}\\
 &=\sup\set[\big]{\abs{\dual{\kispair{Ax_1}{Bx_1}}{\kispair{0}{y_2}}}^2 }{x_1\in E_1, \dual{\kispair{Ax_1}{Bx_1}}{\kispair{x_1}{0^{}_{}}}\leq 1}\\
 &=\sup\set[\big]{\abs{\dual{Bx_1}{y_2}}^2 }{x_1\in E_1, \dual{Ax_1}{x_1}\leq 1}.
\end{align*}
On the other hand,
\begin{align*}
\dual{SS^*y_2}{y_2}&=\sipa{S^*y_2}{S^*y_2}\\
&=\sup\set[\big]{\abs{\sipa{Ax_1}{S^*y_2}}^2}{x_1\in E_1, \sipa{Ax_1}{Ax_1}\leq 1}\\
&=\sup\set[\big]{\abs{\dual{S_0(Ax_1)}{y_2}}^2}{x_1\in E_1, \dual{Ax_1}{x_1}\leq 1}\\
&=\sup\set[\big]{\abs{\dual{Bx_1}{y_2}}^2}{x_1\in E_1, \dual{Ax_1}{x_1}\leq 1}
\end{align*}
Thus we conclude that the quadratic forms of  $SS^*$ and $C_N$ are identical, so $SS^*=C_N$. Finally, applying  \eqref{L:JJquadratic2} we obtain that 
\begin{align*}
\dual{C_Ny}{y}&=\sup\set[\big]{\dual{\kispair{Ax_1}{Bx_1}}{\kispair{0}{y_2}}+\overline{\dual{\kispair{Ax_1}{Bx_1}}{\kispair{0}{y_2}}}-\dual{\kispair{Ax_1}{Bx_1}}{\kispair{x_1}{0^{}_{}}}}{x_1\in E_1}\\
&=\sup\set[\big]{\dual{Bx_1}{y_2}+\dual{B^*y_2}{x_1}-\dual{Ax_1}{x_1} }{x_1\in E_1}.
\end{align*}
The proof is complete.
\end{proof}

Now we are able to introduce the notion of complement in this setting.
\begin{definition}
We refer to the operator $A_B$ in \eqref{Schurdef} as the complement of $A$ with respect to $B$. If $C\in\mathscr{L}(E_2;F_2)$ is any positive operator that makes the system $\kismatrix{A}{B^*}{B}{C}$ positive, then $C-A_B$ is called the Schur complement of $C$ in the block matrix $\kismatrix{A}{B^*}{B}{C}$.
\end{definition}

In the sequel we work with the complement only in the special case when $A,B\in\lef$ are both positive (and hence self-adjoint) operators on $\dual{F}{E}$. In that case the quadratic form of $A_B$ can be calculated as
 \begin{equation}\label{E:cofform}
\dual{A_By}{y}=\sup\set[\big]{\dual{By}{x}+\dual{Bx}{y}-\dual{Ax}{x} }{x\in E}.
\end{equation}

As a straightforward consequence of Theorem \ref{T:Schur} we retrieve the classical result  \cite[$\S$1 Condition 2' and Theorem 1.1]{Pekarev} of Pekarev and \v{S}mul'jan.

\begin{corollary}
Let $\hil$ be a Hilbert space and let $A,B\in\bh$ be bounded operators. Assume further that $A$ is positive. Then the following assertions are equivalent:
\begin{enumerate}[label=\textup{(\roman*)}]
 \item There exists a $C\in\bh$ such that $\kismatrix{A}{B^*}{B}{C}$ is positive,
 \item  $B^*B\leq mA$ for some constant $m\geq0$,
 \item $\ran B^*\subseteq \ran A^{1/2}$.
\end{enumerate}
\end{corollary}
Before proceeding with direct applications let us make first two short comments. On the one hand, if we consider a sesquilinear form $\tfrom:E\times E\to\mathbb{C}$ on a given complex vector space $E$, then we may associate a linear operator $A_{\tform}$ acting between $E$ and its algebraic anti-dual $\bar{E}'$ by setting
\begin{align}\label{E:opform}
\dual{A_{\tform}x}{y}:=\tform(x,y),\qquad x,y\in E.
\end{align}
It is clear now that $A_{\tform}$ is automatically weakly continuous that is positive (resp., selfadjoint) if and only if $\tform$ is nonnegative semidefinite (resp., Hermitian). On the other hand, if an anti-dual pair $\dual FE$ and a positive operator $A\in\lef$ is given,  we can associate a hermitian form on $E$ by
\begin{align}\label{E:opformvissza}
\tform_A(x,y):=\dual{Ax}{y},\qquad x,y\in E.
\end{align}
Therefore, using the above theorem, we can define the complement of sesquilinear forms. This notion (together with the parallel sum and parallel difference) was introduced and studied by Hassi, Sebesty\'en, and de Snoo in \cite{HSdS} and \cite{HSdS2}.
\subsection{Parallel sum, parallel difference, and Lebesgue decomposition}\label{Ss:PsPdLeb}

As an application of Theroem \ref{T:Schur}, we introduce first the notion of parallel sum in context of positive operators on an anti-dual pair. Consider the incomplete operator matrix $\kismatrix{A+B}{A}{A}{*}$ and observe immediately that it is positive because $\kismatrix{A+B}{A}{A}{A}$  is positive. In particular, the corresponding complement $(A+B)_A$ satisfies $(A+B)_A\leq A$ and its quadratic form can be calculated due to \eqref{E:cofform} as follows 
\begin{align*}
\dual{{(A+B)}^{}_{A}y}{y}&=\sup\set{\dual{Ax}{y} +\dual{Ay}{x}-\dual{(A+B)x}{x}}{x\in E}\\
                  &=-\inf\set{\dual{Ax}{y} +\dual{Ay}{x}+\dual{Ax}{x}+\dual{Bx}{x}}{x\in E}\\
                  &=\dual{Ay}{y}-\inf\set{\dual{A(y+x)}{y+x}+\dual{Bx}{x}}{x\in E}.
\end{align*}
We gain therefore an equivalent definition of the parallel sum introduced first in this setting in \cite[Theorem 4.1]{Tarcsay-ADP Lebdec}.
\begin{definition} Assume that $A,B\in\lef$ are positive operators. Then the operator \begin{equation}
    A:B:=A-(A+B)_A
\end{equation} called \emph{the parallel sum} of $A$ and $B$ is positive, and it satisfies
\begin{equation}\label{E:A:B}
\dual{(A:B)y}{y}=\inf\set{\dual{A(y+x)}{y+x}+\dual{Bx}{x}}{x\in E}.
\end{equation}
\end{definition}
Note that $A:B$ is the Schur complement of $A$ in the block matrix $\kismatrix{A+B}{A}{A}{A}$ by definition.
Observe also that $A:B=B:A$ which is not clear from the definition but can easily deduced  from  formula \eqref{E:A:B}. 

Consider now the system $\kismatrix{A-B}{A}{A}{*}$. In order to be positive it is necessary (but nut sufficient) that  $A-B\geq0$. By Theorem \ref{T:Schur}, positivity of $\kismatrix{A-B}{A}{A}{*}$ is equivalent with the following condition: for any $y\in E$ there is $M_y\geq0$ such that 
 \begin{equation}\label{E:A-BA}
 \abs{\dual{Ax}{y}}^2\leq M_y\dual{(A-B)x}{x}\qquad x\in E.
 \end{equation}
If this holds true, then the complement $(A-B)_A$ exists and its quadratic form is calculated as follows: 
\begin{align*}%\label{F:(A-B)_A}
\dual{(A-B)_Ay}{y}&=\sup\set{\dual{Ax}{y}+\dual{Ay}{x}-\dual{(A-B)x}{x}}{x\in E}\\
                  &=\dual{Ay}{y}+\sup\set{\dual{Bx}{x}-\dual{A(x+y)}{x+y}}{x\in E}\\
                  &=\dual{Ay}{y}+\sup\set{\dual{B(x+y)}{x+y}-\dual{Ax}{x}}{x\in E}
\end{align*}
Observe that the assumption $A\geq B$ is not sufficient to guarantee the existence of $(A-B)_A$. Indeed, the supremum above can be infinity if for example $A=B$. To see this, substitute $x=\lambda y$ for any $y\in E$ satisfying $\dual{Ay}{y}\neq0$.
\begin{definition} Assume that $A,B\in\lef$ are operators such that $(A-B)_A$ does exist. Then the operator defined by 
\begin{equation}
    B \div A:=(A-B)_A-A
\end{equation}
belongs to $\lef$ and is called the \emph{parallel difference} of $B$ and $A$. The quadratic form of $B\div A$ can be calculated as
\begin{equation}\label{E:BkulA}
\dual{(B\div A) y}{y}=\sup\set{\dual{B(x+y)}{x+y}-\dual{Ax}{x}}{x\in E}.
\end{equation}
\end{definition}

%The importance of \eqref{E:A:B} and \eqref{E:BkulA} will be clear once we combine these formulae with \eqref{E:opformvissza}. It turns out that the Lebesgue decomposition developed in \cite{Tarcsay-ADP Lebdec} can be treated in an alternative way. 
Now we are going to apply the above results to retrieve the Lebesgue decomposition developed in \cite{Tarcsay-ADP Lebdec} in an alternative way, namely, by means of the parallel addition and subtraction.  
To do so we recall first what a Lebesgue decomposition is. For more details, see \cite{Tarcsay-ADP Lebdec}.
\begin{definition}\label{D:LD}
Let $A,B\in\lef$ be positive operators. We say that $B$ is \emph{$A$-absolute continuous} ($A\ll B$) if for any sequence $\seq{x}$ of $E$,
 \begin{equation*}
    \dual{Ax_n}{x_n}\to0\qquad \textrm{and}\qquad \dual{B(x_n-x_m)}{x_n-x_m}\to0\qquad (n,m\to+\infty)
 \end{equation*}
 imply $\dual{Bx_n}{x_n}\to0$. We say that $A$ and $B$ are \emph{mutually singular} ($A\perp B$) if $C\leq A$ and $C\leq B$ imply $C=0$ for any positive operator $C\in\lef$. A decomposition $A=A_1+A_2$ of $A$ is called a \emph{Lebesgue-type decomposition} of $A$ with respect to $B$ if $A_1\ll B$ and $A_2\perp B$.
\end{definition}
%As it is known, the notion of $B$-absolute continuity (which is sometimes called $B$-closability) is equivalent with almost dominatedness by $B$. 
It was proved in \cite[Theorem 4.6]{Tarcsay-ADP Lebdec} that the pointwise limit $A_r\coloneqq \limn A:nB$ is $B$-absolutely continuous  and $A_s\coloneqq A-A_r$ is $B$-singular. Thus $A=A_r+A_s$ is a Lebesgue decomposition of $A$ with respect to $B$. Furthermore, according to \cite[Theorem 5.1]{Tarcsay-ADP Lebdec}, the positive operator $A$ is $B$-absolutely continuous if and only if $A=\limn (A:nB)$.

The key results we are going to take advantage of are \cite[Theorem 3.5]{HSdS} and \cite[Theorem 3.9]{HSdS}. The combination of the two statements (with the notation of \eqref{E:opformvissza}) says that the limit of the weighted parallel sums can be formulated by means of the parallel sum and the parallel difference
\begin{equation}\lim\limits_{n\to\infty}\tform_A:n\tform_B=(\tform_A:\tform_B)\div\tform_B.\end{equation}
Using  \eqref{E:A:B} and \eqref{E:BkulA}, this implies that
\begin{equation}\label{(A:B)-B}
\lim\limits_{n\to\infty}A:nB=(A:B)\div B.
\end{equation}
According to \cite[Theorem 4.6]{Tarcsay-ADP Lebdec}, the left hand side of \eqref{(A:B)-B} is the distinguished $B$-absolute continuous part of $A$, and thus we obtained the following form of the Lebesgue decomposition theorem.
\begin{theorem}\label{T:ADP-Lebdec}
Let $\dual FE$ be a weak-$^*$ sequentially complete anti-dual pair. Assume that $A$ and $B$ are positive operators belonging to $\lef$. Then the operator
$A_r:=(A:B)\div B$
belongs to $\lef$, and the decomposition
\begin{equation*}
A=A_r+(A-A_r)
\end{equation*}
is a Lebesgue decomposition of $A$ with respect to $B$. Furthermore,
\begin{equation}\label{E:Areg}
    A_r=\limn A:nB.
\end{equation}
\end{theorem}
We conclude this section with a corollary which states that $A_r$  can be expressed by means of the complement and the parallel sum.
\begin{corollary}\label{C:Nishio}
Let $\dual FE$ be a weak-$^*$ sequentially complete anti-dual pair. Assume that $A$ and $B$ are positive operators belonging to $\lef$. Then 
\begin{equation}
    A_r=(B-B:A)_{B}-B.
\end{equation}
\end{corollary}
\begin{proof}
First observe that if $C$ and $A$ are positive operators such that $C_B$ does exits, then it can be written as
\begin{equation}\label{E:compalt}
C_B=B+(B-C)\div B.
\end{equation}
Indeed, using that $\dual{Bx}{y}+\dual{By}{x}=\dual{Bx}{x}+\dual{By}{y}-\dual{B(x-y)}{(x-y)}$, an elementary calculation shows that
\begin{align*}
    \dual{C_Bx}{x}&=\sup\set{\dual{Bx}{y}+\dual{By}{x}-\dual{Cy}{y}}{y\in E}\\
    &=\dual{Bx}{x}+\sup\set{\dual{(B-C)y}{y}-\dual{B(x-y)}{x-y}}{y\in E}\\
    &=\dual{(B+(B-C)\div B)x}{x}.
\end{align*}
Substituting $C=B-A:B$ in \eqref{E:compalt}, we get
\begin{equation*}
    (B-A:B)_B=B+(A:B)\div B.
\end{equation*}
After rearrangement we conclude that $(B-A:B)_B-B=(A:B)\div B=A_r.$
\end{proof}
\section{Applications}\label{S:Appl}

\subsection{Operators on rigged Hilbert spaces}\label{Ss:rigged}
Let $\D$ be a dense linear subspace of a complex Hilbert space $\hil$ and denote by $\mathcal{L}^{\dagger}(\D)$ the set of those closable operators which leave, together with their adjoints,  $\D$ invariant:
\begin{equation*}
    \mathcal{L}^{\dagger}(\D)\coloneqq\set{A:\D\to\hil}{\D\subseteq \dom A^*, A(\D)\subseteq \D, A^*(\D)\subseteq \D}.
\end{equation*} 
Clearly, the involution $A^{\dagger}\coloneqq A^*|_{\D}$ makes  $\mathcal{L}^\dagger(\D)$  a $^*$-algebra. Consider a $\dagger$-closed subalgebra $\alg$ of $ \mathcal{L}^{\dagger}(\D)$ that contains the identity operator  and equip $\D$ with a locally convex topology $\tau^{}_{\alg}$ induced by the semi-norms
\begin{equation*}
    |||x|||^{}_A\coloneqq (\|x\|^2+\|Ax\|^2)^{1/2},\qquad A\in\alg, x\in \D.
\end{equation*}
Clearly, $\tau^{}_{\alg}$ is finer then the norm topology induced by $\hil$ on $\D$. The canonical embedding of $\D$ into $\hil$ is continuous with dense range and therefore the adjoint of that map embeds $\hil$ into $\DD$ continuously, when endowed with the strong dual topology $\beta'\coloneqq \beta(\DD,\D)$:
\begin{equation*}
    (\D,\tau^{}_{\alg}) \hookrightarrow (\hil,\|\cdot\|) \hookrightarrow (\DD,\beta').
\end{equation*}
This triplet $(\D,\hil,\DD)$ is usually called a rigged Hilbert space or Gelfand triplet. For more details about rigged Hilbert spaces  see \cite{Antoine} and  \cite[Chapter 3]{Schmudgen2}. Suppose now that $(\D,\tau_\alg)$ is a barrelled space (for example, a Fr\'echet space). Every weakly continuous (and in particular, every positive) operator $A:\D\to\DD$ is then continuous with respect to the topologies $\tau_\alg$ and $\beta'$. On the converse, every linear operator $A$ that is continuous for $\tau_\alg$ and $\beta'$ is weakly continuous. In any case, we shall call $A$ simply continuous. Furthermore, since $\DD$ is weakly quasi-complete, we infer that the anti-dual pair $\dual{\DD}{\D}$ is weak-* sequentially complete. This  enables us to apply our preceding results in this rigged Hilbert space framework.

\begin{theorem}\label{T:riggedSchur}
Let $(\D,\hil,\DD)$ be a rigged Hilbert space and suppose that $(\D,\tau_\alg)$ is a barrelled space. Let $A,B:\D\to\DD$ be continuous linear operators, such that $A$ is positive. The following assertions are equivalent: 
\begin{enumerate}[label=\textup{(\roman*)}]
    \item There is a positive operator $C:\D\to\DD$ such that $\kismatrix{A}{B^*}{B}{C}$ is positive
    \item For every $y\in \D$ there is a constant $M_y\geq0$ such that 
    \begin{equation*}
         \abs{\dual{Bx}{y}}^2\leq M_y\cdot \dual{Ax}{x},\qquad \mbox{for all $x\in \D$}.
     \end{equation*}
\end{enumerate}
If any of (i) or (ii) is satisfied, there exists the smallest positive operator $A_B:\D\to\DD$ that makes the matrix $\kismatrix{A}{B^*}{B}{A_B}$ positive.
\end{theorem}

The above theorem enables us to introduce the parallel sum of positive operators in the rigged Hilbert space context: under the assumptions of Theorem \ref{T:riggedSchur}, for every pair $A,B\in \mathscr{L}(\D;\DD)$ of positive operators there exists a unique positive  operator $A:B$, called the parallel sum of $A$ and $B$ with quadratic form \eqref{E:A:B}. If $A$ and $B$ satisfy \eqref{E:A-BA} then there exists a unique positive operator $B\div A$, called the parallel difference of $B$ and $A$ with quadratic form \eqref{E:BkulA}. 

Using the concepts of parallel sum and complement we can easily establish a Lebesgue-type decomposition theorem. To this aim we recall first the concepts of regular and singular operators (cf. \cite[Definition 3.1-3.3]{dibella}).
\begin{definition}
We say that the positive operator $A\in \mathscr{L}(\D;\DD)$ is 
\begin{enumerate}[label=\textup{(\alph*)}]
    \item \emph{regular} if for every sequence $\seq x$ in $\D$,
\begin{equation*}
    \|x_n\|\to0 \qquad \mbox{and} \qquad \dual{A(x_n-x_m)}{x_n-x_m}\to0,\qquad (n,m\to+\infty)
\end{equation*}
imply $\dual{Ax_n}{x_n}\to0$,
\item \emph{singular} if for every $x\in \D$ there exists a sequence $\seq x$ in $\D$ such that 
\begin{equation*}
    \|x_n-x\|\to0\qquad \mbox{and}\qquad \dual{Ax_n}{x_n}\to 0, \qquad(n\to+\infty).
\end{equation*}
\end{enumerate}
\end{definition}
Let us denote the natural inclusion operator $\D\hookrightarrow\DD$ by $I_{\D,\DD}$:  $$I_{\D,\DD}:\D\to\DD,\qquad  x\mapsto \sip{x}{\cdot}\equiv \dual{x}{\cdot}.$$ Clearly, $I_{\D,\DD}$ is a positive operator.  It was proved in \cite[Theorem 6.1 (vii)]{Tarcsay-ADP Lebdec} that $B$ is a singular operator if and only if $B$ and $I_{\D,\DD}$ are mutually singular in the sense of Definition \ref{D:LD}. And by definition, $A$ is regular if and only if $A$ is absolutely continuous with respect to $I_{\D,\DD}$. 

As a consequence of Corollary \ref{C:Nishio}, we are in the position to extend the recent Lebesgue-type decomposition theorem \cite[Theorem 8.3]{Tarcsay-ADP Lebdec}. For an analogous result we refer the reader to the recent result \cite[Theorem 3.6]{dibella}.
\begin{theorem} Let $(\D,\hil,\DD)$ be a rigged Hilbert space and assume that $(\D,\tau_\alg)$ is a barrelled space. Then for every positive operator $A\in\mathscr{L}(\D;\DD)$ the operator
\begin{equation}A_r=(I_{\D,\DD}-I_{\D,\DD}:A)_{I_{\D,\DD}}-I_{\D,\DD}
\end{equation}
is regular, and the decomposition
\begin{equation}A=A_r+(A-A_r)
\end{equation}
is a Lebesgue-type decomposition. That is, $A_r$ is regular, $A-A_r$ is singular. 
\end{theorem}

\subsection{Representable functionals}\label{Ss:functionals}
We have already seen some immediate applications of the complement in the operator context. The aim of this section is to show how this general setting can be used in the theory of representable functionals. First we fix the terminology.

Let $\alg$ be a not necessarily unital complex $^*$-algebra. A linear functional $f$ on $\alg$ is called \emph{representable} if there is GNS triple, i.e. a Hilbert space $\hil_f$, a *-homomorphism $\pi_f:\mathscr{A}\to\bhf$ and  a vector $\xi_f\in\hilf$ such that
$$f(a)=\sip{\pi_f(a)\xi_f}{\xi_f}_f,\qquad a\in\mathscr{A}.$$
As it is known,   \eqref{E:representable} and $\eqref{E:bounded}$ below are equivalent conditions of  the representability of a positive functional $f$ (see e.g. Palmer \cite[Theorem 9.4.15]{palmer}): there exists a  constant $C \geq0$ such that   
\begin{equation}\label{E:representable}
    \abs{f(a)}^2\leq C\cdot f(a^*a),\qquad a\in\alg,
  \end{equation}
  and for every $a\in\alg$ there exists $\lambda_a\geq0$ such that
  \begin{equation}\label{E:bounded}
    f(b^*a^*ab)\leq \lambda_a \cdot f(b^*b),\qquad b\in\alg.
  \end{equation}
  The set of representable functionals will be denoted by $\alg^{\sharp}$. The GNS triple mentioned above can be obtained by the following construction: 
since every representable functional $f$ is positive, the map $A:\alg\to \bar \alg^*$ defined by 
\begin{equation*}
    \dual{Aa}{b}\coloneqq f(b^*a),\qquad a,b\in\alg
\end{equation*}
is a positive operator. The range space $\ran A$ becomes a pre-Hilbert space if we endow it by the inner product
  \begin{equation*}%\label{E:sipa}
    \sipf{Aa}{Ab}:=f(b^*a),\qquad a\in\alg.
  \end{equation*}
This is indeed an inner product space as the Cauchy--Schwarz inequality 
  \begin{equation*}\label{E:CBS}
    \abs{f(b^*a)}^2\leq f(a^*a)f(b^*b), \qquad a,b\in\alg.
  \end{equation*}
guarantees that $\sipf{Aa}{Aa}=0$ implies $Aa=0$ for all $a\in\alg$. We denote the Hilbert space completion of this space by $\hilf$. Next we introduce a densely defined continuous operator $\pif(x)$ for all $x\in\alg$ by 
  \begin{equation*}%\label{E:pia}
    \pif(x)(Aa):=A(xa),\qquad a\in\alg.
  \end{equation*}
  The continuity of $\pi_f(x)$ is due to \eqref{E:bounded}, and we continue to write $\pif(x)$ for its unique norm preserving extension. It is easy to verify that $\pif$ is a $^*$-representation of $\alg$ in $\bhf$. The cyclic vector of $\pif$ is obtained by considering the linear functional $Aa\mapsto f(a)$ from $\hilf$ into $\dupC$ whose continuity is guaranteed by \eqref{E:representable}. The Riesz representation theorem yields a unique vector $\xi_f\in\hilf$  satisfying
  \begin{equation}\label{E:zeta_A}
    f(a)=\sipf{Aa}{\xi_f},\qquad a\in\alg.
  \end{equation}
  It is straightforward to verify that
  \begin{equation}\label{E:piazeta}
    \pi_f(a)\xi_f=Aa,\qquad a\in\alg,
  \end{equation}
   whence we infer that
  \begin{equation}\label{E:cyclic}
    f(a)=\sipf{\pif(a)\xi_f}{\xi_f},\qquad a\in\alg.
  \end{equation}

The following is the main result of this section. It provides a sufficient condition to the existence of the complement of functionals.
\begin{theorem}\label{T:Schur_funct}
Let $f,g$ be linear functionals on a $^*$-algebra $\alg$. Suppose that $f$ is representable and that there is a constant $C\geq 0$ such that 
\begin{equation}\label{E:Schur_fuctional}
\abs{g(a)}^2\leq Cf(a^*a),\qquad a\in\alg.
\end{equation}
Then there exists a representable positive functional $h$ such that $f+g+g^*+h$ is representable and 
\begin{equation}\label{E:egh0}
f(a^*a)+g(b^*a)+\overline{g(b^*a)}+h(b^*b)\geq0
\end{equation}
for all $a,b$ in $\alg$. Furthermore, there is a smallest $h$ possessing this property.
\end{theorem}
\begin{proof}
Consider the following linear functional on $\hilf$ 
\begin{align*}
\pi_f(a)\xi_f\mapsto g(a),\qquad a\in\alg.
\end{align*}
It is bounded according to \eqref{E:Schur_fuctional}, hence there exists a unique $\eta_g\in\hilf$ such that 
\begin{equation}
    g(a)=\sipf{\pi_f(a)\xi_f}{\eta_g},\qquad a\in \alg.
\end{equation}
Let us define $h$ by setting
\begin{equation}\label{E:h}
h(a)=\sipf{\pi_f(a)\eta_g}{\eta_g},\qquad a\in\alg.
\end{equation}
Clearly, $h\in \alg^{\sharp}$. We claim that $h$ is the smallest representable functional with property \eqref{E:egh0}. 
To see this let us consider  the operators $A,B\in\mathscr{L}(\alg;\bar\alg')$ defined by $\dual{Aa}{b}:=f(b^*a)$ and $\dual{Ba}{b}:=g(b^*a)$ ($a,b\in\alg$). Then the incomplete matrix $\kismatrix{A}{B^*}{B}{*}$ is positive as it fulfills condition (ii) of Theorem \ref{T:Schur}:
\begin{align*}
\abs{\dual{Ba}{b}}^2&=\abs{g(b^*a)}^2\leq Cf(a^*bb^*a)=\|\pi_f(b^*a)\xi_f\|^2_f\\
&\leq \|\pi_f(b^*)\|^2  \|\pi_f(a)\xi_f\|^2_f=\|\pi_f(b^*)\|^2\dual{Aa}{a},
\end{align*}
$a,b\in\alg$. By Theorem \ref{T:Schur}, there exists a positive operator  $A_B:\alg\to\bar\alg'$ such that $\kismatrix{A}{B^*}{B}{A_B}$ is positive, furthermore  the quadratic form of $A_B$ is calculated as follows:
\begin{align*}
\dual{A_Ba}{a}&=\sup\set{\abs{\dual{Bb}{a}}^2}{b\in \alg, \dual{Ab}{b}\leq 1}\\
            &=\sup\set{\abs{g(a^*b)}^2}{b\in \alg, \dual{Ab}{b}\leq 1}\\
            &=\sup\set{\abs{\sipf{\pi_f(a^*b)\xi_f}{\eta_g}}^2}{b\in \alg, \|\pi_f(b)\|^2_f\leq 1}\\
            &=\sup\set{\abs{\sipf{\pi_f(b)\xi_f}{\pi_f(a)\eta_g}}^2}{b\in \alg, \|\pi_f(b)\|^2_f\leq 1}\\
            &=\|\pi_f(a)\eta_g\|^2_f\\
            &=h(a^*a).
\end{align*}
Hence, for $a,b$ in $\alg$,
\begin{align*}
f(a^*a)+\overline{g(b^*a)}+g(b^*a)+h(a^*a)=\dual[\big]{\kismatrix{A}{B^*}{B}{A_B}\kispair{a}{b}}{\kispair{a}{b}}\geq 0,
\end{align*}
which shows that $h$ satisfies  \eqref{E:egh0}. To show the minimality of $h$ suppose $h'\in\alg^{\sharp}$ fulfills \eqref{E:egh0} as well. Then letting 
\begin{align*}
\dual{C'a}{b}:=h'(b^*a),\qquad a,b\in\alg,
\end{align*}
one gets $\kismatrix{A}{B^*}{B}{C'}\geq0$ and thus $A_B\leq C'$ by Theorem \ref{T:Schur}, which is equivalent to $h\leq h'.$ Finally, a direct calculations shows that 
\begin{align*}
\sipf{\pi_f(a)(\xi_f+\eta_g)}{\xi_f+\eta_g}=f(a)+g(a)+\overline{g(a^*)}+h(a),\qquad a\in\alg,
\end{align*}
which proves that $f+g+g^*+h\in\alg^{\sharp}$.  
\end{proof}
\begin{definition}
We call the smallest representable functional satisfying \eqref{E:egh0} the \emph{complement of $f$ with respect to $g$}, and we  denote it by $f_g$. The complement $f_g$ can be calculated as
\begin{align*}
f_g(a)=\sipf{\pi_f(a)\eta_g}{\eta_g},\qquad a\in\alg,
\end{align*}
where $\eta_g\in \hilf$ is the representing vector of the bounded linear functional $$\hilf\to\dupC;\qquad \pi_f(a)\xi_f\mapsto g(a).$$
\end{definition}
With the notation of the proof of Theorem \ref{T:Schur_funct}, we  have also that
\begin{align*}
f_g(a^*a)=\dual{A_Ba}{a},\qquad a\in \alg.
\end{align*}
Hence we gain two formulae for $f_g(a^*a)$: from \eqref{C:Cor32} we get 
\begin{equation}\label{E:f_g(a*a)1}
f_g(a^*a)=\sup\set{\abs{g(a^*b)}^2}{b\in \alg, f(b^*b)\leq 1},
\end{equation}
and similarly, \eqref{2cof} yields
\begin{equation}\label{E:f_g(a*a)2}
f_g(a^*a)=\sup\set{g(a^*b)+\overline{g(a^*b)}-f(b^*b)}{b\in\alg}.
\end{equation}
Recall also that, with notation of Theorem \ref{T:Schur}, we have $A_B=SS^*$, where $S:\hila\to F$ is the unique weakly continuous operator defined by 
\begin{align*}
S(Ax)=Bx,\qquad x\in E.
\end{align*} 
This operator plays a key role in the following simple corollary, which provides an elegant formula for the complement in the case when $\alg$ possesses a unit element.

\begin{corollary}Let $f,g$ be linear functionals on a $^*$-algebra $\alg$. Suppose that $f\in\alg^{\sharp}$ and $g$ satisfies \eqref{E:Schur_fuctional}. 
Then the complement $f_g$ of $f$ with respect to $g$ is of the form $f_g=\overline{S\eta_g}$, i.e.
\begin{equation}\label{E:fg=Seta}
f_g(a)=\overline{\dual{S\eta_g}{a}},\qquad a\in \alg.
\end{equation}
If $\alg$ is unital, then $f_g=\overline{A_B1}$, i.e.,
\begin{equation}\label{E:fg=AB1}
f_g(a)=\overline{\dual{A_B1}{a}},\qquad a\in\alg.
\end{equation}
\end{corollary}
\begin{proof}
First observe that $S^*a=\pi_f(a)\eta_g$ holds for all $a\in\alg$. Indeed, according to Theorem \ref{T:Schur_funct}, we have that $g(a)=\sipf{\pi_f(a)\xi_f}{\eta_g}$, and thus
\begin{align*}
\sipa{Afx}{S^*a}&=\dual{S(Ax)}{a}=\dual{Bx}{a}=g(a^*x)\\
               &=\sipf{\pi_f(a^*x)\xi_f}{\eta_g}=\sipf{\pi_f(x)\xi_f}{\pi_f(a)\eta_g}=\sipf{Ax}{\pi_f(a)\eta_g}.
\end{align*}
Using this, we easily conclude that
\begin{align*}
\overline{\dual{S\eta_g}{a}}=\sipf{S^*a}{\eta_g}=\sipf{\pi_f(a)\eta_g}{\eta_g}=f_g(a),
\end{align*}
which implies \eqref{E:fg=Seta}. Furthermore,  $S^*a=\pi_f(a)\eta_g$ implies 
\begin{align*}
\dual{A_Bx}{a}&=\dual{SS^*x}{a}=\sipf{S^*x}{S^*a}\\
              &=\sipf{\pi_f(x)\eta_g}{\pi_f(a)\eta_g}=\sipf{\eta_g}{\pi_f(x^*a)\eta_g}=\overline{f_g(x^*a)},
\end{align*}
which obviously implies \eqref{E:fg=AB1} whenever $1\in\alg$.
\end{proof}

As we have seen in the previous section, the existence of the complement guarantees the existence of some important operations like the parallel sum and parallel difference. Theorem \ref{T:Schur_funct} now allows us to extend these concepts for representable functionals.

We begin with the parallel sum. Consider two representable functionals $f,g \in\alg^{\sharp}$, then $f+g\in\alg^{\sharp}$ and
\begin{align*}
\abs{f(a)}^2\leq M(f(a^*a)+g(a^*a))\qquad a\in\alg
\end{align*}
holds for suitable $M$. Thus the complement $(f+g)_f\in\alg^\sharp$ exists and as
\begin{multline*}
f(a^*a)+g(a^*a)+f(b^*a)+\overline{f(b^*a)}+f(b^*b)=f((a+b)^*(a+b))+g(a^*a)\geq0,
\end{multline*}
we conclude from Theorem \ref{T:Schur_funct} that $(f+g)_f$ exists and $f\geq (f+g)_f$. Therefore the following definition is correct.

\begin{definition}
Let $f$ and $g$ be representable functionals on the ${}^*$-algebra $\alg$. Then the parallel sum $f:g$ defined by $(f:g):=f-(f+g)_f$ is a representable functional as well. From  \eqref{E:f_g(a*a)2} it is easy to check that $f:g$ satisfies 
\begin{equation}\label{E:par}
(f:g)(a^*a)=\inf\set{f((a+b)^*(a+b))+g(b^*b)}{b\in\alg}.
\end{equation}
\end{definition}
For a completely different approach to the parallel sum see \cite{Tarcsay_parallel}. Now we proceed with the notion of parallel difference.

\begin{definition}
Let $f$ and $g$ be representable functionals on the ${}^*$-algebra $\alg$, and assume that $f\geq g$ and the complement $(f-g)_f\in\alg^{\sharp}$ exists. Then the parallel difference $g\div f\in\alg^{\sharp}$ is defined by $g\div f:=(f-g)_f-f$.
\end{definition}
It is easy to deduce from \eqref{E:f_g(a*a)2} that $(g\div f):=(f-g)_f-f$ satisfies 
\begin{equation}\label{E:pardiff_funct}
(g\div f)(a^*a)=\sup\set{g((a+b)^*(a+b))-f(b^*b)}{b\in\alg}.
\end{equation}
As it was pointed out in the operator setting, the assumption $f\geq g$ is not enough for the existence of $g\div f$. However, Theorem \ref{T:Schur_funct} offers a sufficient condition as follows:  assume that
\begin{align}\label{E:pardiff_felt}
\abs{f(a)}^2\leq C\cdot(f(a^*a)-g(a^*a)),\qquad a \in \alg
\end{align}
holds for some $C\geq0$. Then \eqref{E:pardiff_felt}  implies that $f-g\in \alg^{\sharp}$ and that the complement $(f-g)_f\in\alg^{\sharp}$ exists.  

To conclude the paper, we establish a  Lebesgue decomposition theorem for representable functionals by means of the parallel sum and difference.  
\begin{definition}
Let $f$ and $g$ be representable functionals on the ${}^*$-algebra $\alg$. We say that \emph{$f$ is $g$-absolutely continuous} ($f\ll g$) if for every sequence $\seq a$ of $\alg$ such that $g(a_n^*a_n)\to0$ and $f((a_n-a_m)^*(a_n-a_m))\to0$ it follows that $f(a_n^*a_n)\to0$. Furthermore, $f$ and $g$ are \emph{singular} ($f\perp g$) with respect to each other if $h=0$ is the only representable functional such that $h\leq f$ and $h\leq g$. A decomposition $f=f_1+f_2$ is called a \emph{Lebesgue-type decomposition of $f$ with respect to $g$} if $f_1\ll g$ and $f_2\perp g$. 
\end{definition}
For the sake of simplicity, let us introduce the following temporary notation: if a representable functional $h$ is given, we denote its induced operator by $A_h$.  Accordingly, it is obvious that $A_f:A_g=A_{f:g}$ and $A_f\div A_g=A_{f\div g}$ due to \eqref{E:par} and \eqref{E:pardiff_funct}. This simple observation leads us to the following Lebesgue type decomposition theorem (for earlier versions of the Lebesgue decomposition of representable functionals see \cite{Gudder, Kosaki, Tarcsay-funcdec, Tarcsay-ADP Lebdec, Glasgow, pems}, and the references therein).
\begin{theorem}
Let $\alg$ be a ${}^*$-algebra, and assume that $f,g\in\alg^{\sharp}$. Then the functional $f_r=(f:g)\div g$ is representable and $g$-absolutely continuous. Furthermore, the decomposition
\begin{equation}
    f=f_r+(f-f_r)
\end{equation}
is a Lebesgue-type decomposition of $f$ with respect to $g$.
Here the absolutely continuous part $f_r$ can be written as
\begin{equation}
    f_r=(g-g:f)_{g}-g.
\end{equation}

\end{theorem}
%\section*{References}


\begin{thebibliography}{10}

\bibitem{Antoine}
J.-P. Antoine and C. Trapani,  
\newblock Partial Inner Product Spaces - Theory and Applications, Springer Lecture Notes in Mathematics, vol. 1986, Berlin, Heidelberg (2009).

\bibitem{BaidiukHassi}
D. Baidiuk and S. Hassi, 
\newblock Completion, extension, factorization, and lifting of operators, Math.
Ann. 364 (2016), no. 3-4, 1415--1450.

\bibitem{B}
D. Baidiuk,
\newblock Completion and extension of operators in Krein spaces, Journal of  Mathematical Sciences, 2017, Volume 224, Number 4, 493--508.

\bibitem{CMM}
M. Contino, A. Maestripieri, and S. Marcantognini,
\newblock Schur complements of selfadjoint Krein space operators, Linear Algebra and its Applications, 2019, Volume 581,  214-246.

\bibitem{dibella}
S. Di Bella and C. Trapani,
\newblock Singular perturbations and operators
in rigged Hilbert spaces, Mediterr. J. Math. 13, (2016) 2011--2024.

\bibitem{FGK}
J. Friedrich, M. G\"unther, and L. Klotz,
\newblock A generalized Schur complement for nonnegative operators on linear spaces, Banach J. Math. Anal. Volume 12, Number 3 (2018), 617--633.

\bibitem{Gudder}
S. P. Gudder,
\newblock A Radon-Nikodym theorem for $^*$-algebras, Pacific J. Math., 80 (1979), 141--149.


\bibitem{HSdS}
S. Hassi, Z. Sebesty{\'e}n, and H. de Snoo, 
\newblock Lebesgue type decompositions for nonnegative forms,
Journal of Functional Analysis, 257 (2009), 3858--3894.

\bibitem{HSdS2}
S. Hassi, Z. Sebesty\'en, and H. de Snoo,
\newblock Domain and range descriptions for adjoint relations, and parallel
sums and differences of forms, Recent advances in operator theory in Hilbert and Krein spaces. Oper. Theory Adv. Appl. 198, 211–227, Birkhäuser Verlag, Basel (2010)

\bibitem{S2} 
E.V. Haynsworth, 
\newblock Determination of the inertia of a partitioned Hermitian
matrix, Linear Algebra and its Applications, 1 (1968), 73--81.

\bibitem{Kosaki}
H. Kosaki,
\newblock Lebesgue decomposition of states on a von Neumann algebra, American J. Math., 107 (1985), 697--735.

\bibitem{MMP}
 A. Maestripieri and F. Martinez Peria, 
 \newblock Schur complements in Krein spaces, Integr. Equ. Oper. Theory,
59 (2007), 207--221.

\bibitem{palmer}
T. W. Palmer, 
\newblock Banach Algebras and the General Theory of $^*$-Algebras II, Cambridge University Press (2001).

\bibitem{Pekarev}
E. L. Pekarev and Ju. L. \v{S}mul'jan,
\newblock Parallel addition and parallel subtraction of operators, Izvestiya: Mathematics, 10 (1976), 351--370.

\bibitem{Schmudgen2}
K. Schm\"udgen,
\newblock Unbounded operator algebras and Representation theory, Birkh\"auser Verlag, Basel, 1990.

\bibitem{Tarcsay-funcdec}
Zs. Tarcsay, 
\newblock Lebesgue decomposition for representable functionals on *-algebras, Glasgow Math. Journal, 58 (2016), 491--501.

\bibitem{Tarcsay_parallel}
Zs. Tarcsay,
\newblock On the parallel sum of positive operators, forms, and functionals, Acta Math. Hungar., 147 (2015), 408--426.

\bibitem{Tarcsay-ADP Lebdec}
Zs. Tarcsay, 
\newblock Operators on anti-dual pairs: Lebesgue decomposition of positive operators, Journal of Mathematical Analysis and Applications, 848 (2) (2020), 123753.


\bibitem{Glasgow}
Zs. Tarcsay and T. Titkos, 
\newblock On the order structure of representable functionals, Glasgow Math. Journal,  60 (2018), 289-305.

\bibitem{KV}
Zs. Tarcsay and T. Titkos, 
\newblock Operators on anti-dual pairs: Generalized Krein--von Neumann extension, arXiv:1810.02619


\bibitem{pems}
T. Titkos, 
\newblock Arlinskii's iteration and its applications, Proc. Edinb. Math. Soc.,  62 (2019), no. 1, 125--133.

\end{thebibliography}
\end{document}